\documentclass[11pt]{article}
\usepackage[a4paper,margin=27mm]{geometry}
\usepackage{amsmath,amssymb,amsthm,mathtools}
\usepackage[T1]{fontenc}
\usepackage{lmodern,microtype}
\usepackage[colorlinks=true,linkcolor=blue,citecolor=blue,urlcolor=blue]{hyperref}
\usepackage{enumitem}
\allowdisplaybreaks[1]
\numberwithin{equation}{section}
\newtheorem{theorem}{Theorem}[section]
\newtheorem{proposition}[theorem]{Proposition}
\newtheorem{lemma}[theorem]{Lemma}
\newtheorem{corollary}[theorem]{Corollary}
\theoremstyle{definition}\newtheorem{definition}[theorem]{Definition}
\theoremstyle{remark}\newtheorem{remark}[theorem]{Remark}
\newcommand{\D}{\mathcal O}
\newcommand{\E}{\mathbb E}
\newcommand{\Prob}{\mathbb P}
\newcommand{\R}{\mathbb R}
\newcommand{\ind}{\mathbf 1}

\newcommand{\dd}{\,\mathrm d}
\DeclareMathOperator{\diver}{div}

\begin{document}
\title{Global Well-posedness and Asymptotic Extinction for Keller--Segel Systems with Superlinear Noise}
\author{Yunfeng Chen$^{1}$,
	Tusheng Zhang$^{2,3}$}
\footnotetext[1]{\, School of Mathematics, Hefei University of Technology, Hefei, China. Corresponding author: Yunfeng Chen. Email: \href{mailto:yfchen@hfut.edu.cn}{yfchen@hfut.edu.cn}.}
\footnotetext[2]{\, School of Mathematics, University of Science and Technology of China, Hefei, China.}
\footnotetext[3]{\, Department of Mathematics, University of Manchester, Manchester M13 9PL, United Kingdom. Email: \href{mailto:tusheng.zhang@manchester.ac.uk}{tusheng.zhang@manchester.ac.uk} }
\date{}

\maketitle

\begin{abstract}
It is well known that solutions to deterministic Keller--Segel
chemotaxis systems can blow up in finite time under certain conditions.
To investigate whether noise can prevent such blow-up, we study
parabolic--elliptic and fully parabolic Keller--Segel systems on
smooth bounded domains with pointwise superlinear multiplicative
It\^o noise $\sigma_0u^r\,\mathrm dW_t$, where $W$ is a real
Brownian motion. Our results show that superlinear multiplicative noise can prevent finite-time blow-up in Keller--Segel systems.  For $r>3/2$, we prove global existence and pathwise uniqueness
of nonnegative mild solutions in arbitrary spatial dimension,
for every $\sigma_0>0$. At $r=3/2$, both systems are globally
well-posed in two dimensions for every positive noise intensity;
in arbitrary dimension, the parabolic--elliptic system is global under
$\sigma_0^2\ge2\chi$.  For $r>3/2$ in the fully parabolic case,
the proof combines a logarithmic estimate for the time derivative
of the chemical signal with comparison of the local solution
to the solution of a scalar stochastic differential equation.
The two-dimensional case follows from mass quadratic variation,
deterministic maximal regularity, and localized energy estimates.
We also establish almost-sure $L^1$ extinction and algebraic
decay of fractional mass moments.
\end{abstract}
\begin{samepage}

\end{samepage}

\section{Introduction}\label{sec:intro}

The Keller--Segel system describes the competition between diffusion
and aggregation along a chemical gradient \cite{KS,Horstmann}.
This competition can lead to global regularity or finite-time blow-up,
depending on the model, the spatial dimension, and the initial data.
In two dimensions, Nagai, Senba, and Yoshida \cite{NSY} established
global existence and boundedness under a smallness condition on the
initial mass for a fully parabolic system on bounded domains.
For the parabolic--elliptic Keller--Segel system in $\R^2$ with
Poisson coupling $-\Delta v=u$, Blanchet, Dolbeault, and Perthame
\cite{BDP} established the critical-mass distinction between global
existence and finite-time blow-up under appropriate assumptions on
the initial data.
In higher dimensions, small-data global boundedness for the fully
parabolic system was proved by Winkler \cite{WinklerSmall}, while
finite-time blow-up was established for suitable radially symmetric
data on balls in \cite{WinklerBlow}.
Logistic damping provides an additional mechanism for controlling
aggregation under suitable assumptions \cite{TelloWinkler,Winkler}. Here we investigate global solvability through a stochastic
perturbation of the density equation, without a deterministic source
or a smallness condition on the initial data.

More precisely,  we consider the following system on a smooth bounded domain
$\D\subset\R^d$, with homogeneous Neumann boundary conditions:
\begin{equation}\label{eq:model}
	\begin{cases}
		\mathrm du=[\Delta u-\chi\diver(u\nabla v)]\,\mathrm dt
		+\sigma_0u^r\,\mathrm dW_t,
		&x\in\D,\ t>0,\\
		\tau\partial_t v=\Delta v-v+u,
		&x\in\D,\ t>0,\\
		\partial_\nu u=\partial_\nu v=0,
		&x\in\partial\D,\ t>0.
	\end{cases}
\end{equation}
Here $\chi>0$, $\sigma_0>0$, $r>1$, and either $\tau=0$ or
$\tau>0$.
The functions $u$ and $v$ denote the cell density and the
chemoattractant concentration, respectively.
The stochastic integral is interpreted in the It\^o sense, and
$W$ is a real-valued standard Brownian motion.

For every fixed finite spatial dimension, we obtain
global existence and pathwise uniqueness of nonnegative mild
solutions for both the parabolic--elliptic ($\tau=0$) and fully parabolic
systems ($\tau>0$) when $r>3/2$ and $\sigma_0>0$.
At $r=3/2$, both systems are globally well-posed in two dimensions
for every $\sigma_0>0$; in arbitrary dimension, the
parabolic--elliptic system is well-posed under the sufficient condition
$\sigma_0^2\ge2\chi$.
For $r>3/2$, we also obtain quantitative finite-time tails for the
spatial supremum and finite fractional moments of its time supremum.

It has been shown in the literature that nonlinear noise can prevent finite-time blow-up
in stochastic evolution equations.
Tang and Wang \cite{TangWang} obtained global well-posedness for stochastic aggregation--diffusion equations on $\mathbb R^d$, $d\ge2$, under suitable conditions on norm-dependent nonlinear noise.
For stochastic chemotaxis on two-dimensional bounded domains,
Chen, Zhai, and Zhang \cite{CZZ} established global well-posedness
results for parabolic--elliptic systems, including noise coefficients
of the form $b_k\|u\|_{L^\ell}^{s}u$.
Wang, Li, and Huang \cite{WLH} treated a fully parabolic Keller--Segel system on a two-dimensional bounded domain with similar norm-dependent noise and proved global existence and pathwise uniqueness without a smallness condition on the initial data, under suitable exponent assumptions. 
Related results for nonlinear It\^o noise in the Euler equations
were obtained by Bagnara, Maurelli, and Xu \cite{BMX}.
In the norm-dependent models described above, the noise coefficients depend on global norms
of the solution and are therefore nonlocal in space.
In the present model, the noise amplitude at each spatial point depends only on the local value of the density.

The crucial idea in this paper is to 
compare the local solution of the  Keller--Segel system with the solution of a scalar stochastic
differential equation driven by the same Brownian motion:
$$\dd U=(\chi U^2+a(t)U)\dd t+\sigma_0U^r\dd W_t.$$
For the parabolic--elliptic Keller--Segel system, the identity
$-\chi\Delta v=\chi u-\chi v$ makes a quadratic scalar drift
sufficient. The fully parabolic coupling produces an additional
term that depends on the unknown solution:
\[
-\chi\Delta v=\chi u-\chi v-\chi\tau v_t,
\qquad
a(t)=\chi\tau\|(v_t(t))_-\|_\infty.
\]
where $z_-:=\max\{-z,0\}$.
Controlling this random reaction coefficient $a(t)$ before global existence
is known is the main difficulty.
We control $a$ by combining polynomial spatial estimates before
the amplitude stopping time with a short-time/long-time
decomposition of the signal convolution, 
which yields an estimate with only a logarithmic loss in the stopping level.

For any global nonnegative solution under the framework of
Section~\ref{sec:setting}, with $r>1$ and $\sigma_0>0$,
the density and signal converge almost surely to zero in $L^1$.
We also prove fractional supremum-moment bounds and quantify the decay:
for every $0<s<1$,
\[
\E\|u(t)\|_1^s+\E\|v(t)\|_1^s
\le C_s(1+t)^{-s/(2r-2)},\qquad t\ge0.
\]

The rest of the paper is organized as follows.
Section~\ref{sec:setting} states the framework and main results.
Section~\ref{sec:local} collects the analytic estimates, local
well-posedness, and continuation tools.
Section~\ref{sec:comparison} develops the scalar comparison.
Section~\ref{sec:parabolic-elliptic} treats the parabolic--elliptic
system, while Section~\ref{sec:global} proves the fully parabolic
result for $r>3/2$.
Section~\ref{sec:critical-2d} establishes the two-dimensional
endpoint result.
Section~\ref{sec:mass-proof} derives quantitative finite-time
bounds and proves asymptotic extinction and algebraic decay
of fractional mass moments.
Appendices~\ref{app:stoch} and~\ref{app:local} contain the
stochastic convolution estimates and the local construction,
respectively.

\section{Main results}\label{sec:setting}
\subsection{Setting and solutions}
Let $d\ge1$ be an arbitrary fixed integer, let $\D\subset\R^d$
be a bounded domain with $C^\infty$ boundary, and let $\nu$ denote
its outward unit normal.
Let $(\Omega,\mathcal F,(\mathcal F_t)_{t\ge0},\Prob)$ be a complete
filtered probability space satisfying the usual conditions, and let
$\{W_t\}_{t\ge0}$ be a real-valued standard Brownian motion with
respect to $(\mathcal F_t)_{t\ge0}$.
All stochastic integrals are understood in the It\^o sense.
We consider system~\eqref{eq:model} with $\chi>0$, $\sigma_0>0$,
$r>1$, and either $\tau=0$ or $\tau>0$.

Let $u_0\in C(\overline\D)$ be deterministic and nonnegative.
When $\tau>0$, let $v_0\in C^2(\overline\D)$ be deterministic,
nonnegative, and satisfy $\partial_\nu v_0=0$.
For $\tau=0$, only $u_0$ is prescribed and
\begin{equation}\label{eq:elliptic-initial}
v(t)=Ku(t),\qquad K=(I-\Delta_N)^{-1};\qquad v(0)=Ku_0.
\end{equation}

Fix an auxiliary exponent \(q>\max\{8,4d\}\), for instance \(q=8(d+1)\). This exponent is used in the local construction and spatial regularity estimates.
Let
\[
S(t)=e^{t\Delta_N},
\qquad
B=\frac{I-\Delta_N}{\tau},
\qquad
S_\tau(t)=e^{-tB},
\quad \tau>0.
\]
Unless otherwise specified, all norms are taken over \(\D\); in particular, \(\|\cdot\|_\infty\) denotes the spatial supremum norm.
We use $B$ and $S_\tau$ for the compatible Neumann realizations on $C(\overline\D)$ and $L^q(\D)$.

\begin{definition}[Local mild solution]\label{def:local-solution}
	Let $\zeta$ be a strictly positive stopping time. A \emph{local mild solution}
	of \eqref{eq:model} with lifetime $\zeta$ is an adapted nonnegative process $u$ such that,
	almost surely,
	\[
	u\in C\bigl([0,\zeta);C(\overline{\D})\bigr),
	\]
	together with the associated signal $v$ given by \eqref{eq:elliptic-initial} if $\tau=0$, or by
	\begin{equation}\label{eq:mild-v}
		v(t)=S_\tau(t)v_0+\frac1\tau\int_0^t S_\tau(t-s)u(s)\dd s
	\end{equation}
	if $\tau>0$. Moreover, there exists a non-decreasing sequence of bounded stopping times
	$(\zeta_n)_{n\ge1}$ such that $\zeta_n<\zeta$ and $\zeta_n\uparrow\zeta$ almost surely as $n\rightarrow \infty$,
	and for every $n\ge1$, the following mild identity holds almost surely for all $0\le t\le\zeta_n$:
	\begin{equation}\label{eq:mild-u}
		u(t)
		=
		S(t)u_0
		-\chi\int_0^t
		S(t-s)\diver\!\bigl(u(s)\nabla v(s)\bigr)\,\dd s
		+\sigma_0\int_0^t
		S(t-s)u(s)^r\,\dd W_s.
	\end{equation}
	
	In particular, the solution is called \emph{global} if $\zeta=\infty$ almost surely.
\end{definition}
The divergence term in \eqref{eq:mild-u} is understood in the weak Neumann
sense. The stochastic convolution is initially defined as an
$L^q(\D)$-valued stochastic integral, using the $L^q(\D)$ realization of the
Neumann heat semigroup, and is then identified with its spatially continuous
modification. Accordingly, \eqref{eq:mild-u} is understood as an identity in
$C(\overline{\D})$ after choosing this modification.

\subsection{Global well-posedness}
\begin{theorem}[Global well-posedness for $r>3/2$]\label{thm:global}
Under the framework above, let $r>3/2$ and $\sigma_0>0$.
For every fixed $d\ge1$ and either $\tau=0$ or $\tau>0$,
system~\eqref{eq:model} has a pathwise unique global nonnegative
mild solution. When $\tau>0$, for every $T<\infty$,
\[
v\in C^1\bigl([0,T];C(\overline\D)\bigr)
\qquad\text{almost surely}.
\]
Fix $\lambda\in(0,1)$ and define
\begin{equation}\label{eq:exponents}
\gamma=2r-2>1,\qquad n=\lambda+\gamma,\qquad
p=\frac{n}{\gamma}=1+\frac\lambda\gamma,\qquad
\nu_* =\frac{n}{\gamma-1}.
\end{equation}
For every $T<\infty$, there is a constant $C_{T,\lambda}<\infty$
such that, for all sufficiently large $R$, the elliptic-signal
solution satisfies
\begin{equation}\label{eq:pe-tail}
\Prob\left(\sup_{0\le t\le T}\|u(t)\|_\infty\ge R\right)
\le\frac{C_{T,\lambda}}{(1+R)^\lambda},\qquad \tau=0,
\end{equation}
whereas the fully parabolic solution satisfies
\begin{equation}\label{eq:pp-tail}
\Prob\left(\sup_{0\le t\le T}\|u(t)\|_\infty\ge R\right)
\le C_{T,\lambda}
\frac{1+(1+\log R)^{\nu_*}}{(1+R)^\lambda},\qquad \tau>0.
\end{equation}
\end{theorem}

\begin{theorem}[Global well-posedness for $r=3/2$]\label{thm:critical}
Under the same framework with $r=3/2$, the following assertions hold.
\begin{enumerate}[label=(\roman*),ref=(\roman*),leftmargin=*]
\item\label{item:critical-pe}
If $\tau=0$ and $\sigma_0^2\ge2\chi$, then for every fixed
$d\ge1$ there is a pathwise unique global nonnegative mild solution.
For every $T<\infty$ and every $R>U_0:=1+\|u_0\|_\infty$,
\begin{equation}\label{eq:critical-pe-tail}
\Prob\left(\sup_{0\le t\le T}\|u(t)\|_\infty\ge R\right)
\le\frac{\log(1+U_0)+\chi T}{\log(1+R)}.
\end{equation}
\item\label{item:critical-2d}
If $d=2$, then for either signal equation and every $\sigma_0>0$
there is a pathwise unique global nonnegative mild solution.
When $\tau>0$, for every $T<\infty$,
\[
v\in C^1\bigl([0,T];C(\overline\D)\bigr)
\qquad\text{almost surely}.
\]
\end{enumerate}
\end{theorem}

\subsection{Finite-time bounds}
\begin{corollary}[A tail bound for the accumulated signal coefficient]\label{cor:tight}
	Let $r>3/2$, $\tau>0$, and let $(u,v)$ be the solution of Theorem~\ref{thm:global}. Set
	\[
	\tau_R=\inf\{t\ge0:\|u(t)\|_\infty\ge R\},\qquad
	A_R(T)=\int_0^{T\wedge\tau_R}\|(v_t(t))_-\|_\infty\dd t.
	\]
	Then, for each $T<\infty$,
	\begin{equation}\label{eq:tight}
		\lim_{K\to\infty}\sup_{R>\|u_0\|_\infty}\Prob\left(A_R(T)>K\right)=0.
	\end{equation}
	More quantitatively, for the exponents in \eqref{eq:exponents} and all sufficiently large $L,K$,
	\begin{equation}\label{eq:tight-rate}
		\sup_{R>\|u_0\|_\infty}\Prob\left(A_R(T)>K\right)
		\le C_T\frac{1+(1+\log L)^{\nu_*}}{(1+L)^\lambda}
		+\frac{C_T[1+L^p(1+\log L)^p]}{K^p}.
	\end{equation}
\end{corollary}

\begin{corollary}[Fractional supremum moments]\label{cor:sup-moments}
	Let $r>3/2$ and let $(u,v)$ be the solution of Theorem~\ref{thm:global}.
	For every $T<\infty$ and $s\in(0,1)$,
	\begin{equation}\label{eq:sup-moments}
		\E\left[\sup_{0\le t\le T}\|u(t)\|_\infty^s\right]<\infty.
	\end{equation}
\end{corollary}

\subsection{Extinction and decay}
\begin{theorem}[Asymptotic mass extinction]\label{thm:mass}
Let $(u,v)$ be any global nonnegative mild solution of
\eqref{eq:model} under the framework of Section~\ref{sec:setting},
with $r>1$ and $\sigma_0>0$. In particular, the conclusions apply
to the solutions constructed in Theorems~\ref{thm:global}
and~\ref{thm:critical}. Set
	\[
	M(t)=\|u(t)\|_{L^1(\D)},\qquad M_0=\|u_0\|_{L^1(\D)}.
	\]
	Then
	\begin{equation}\label{eq:mass-extinction}
		\lim_{t\to\infty}\|u(t)\|_{L^1(\D)}
		=\lim_{t\to\infty}\|v(t)\|_{L^1(\D)}=0
		\qquad\text{almost surely},
	\end{equation}
	and
	\begin{equation}\label{eq:mass-integral}
		\int_0^\infty M(t)^{2r}\dd t<\infty
		\qquad\text{almost surely}.
	\end{equation}
	For every $s\in(0,1)$ one also has
	\begin{equation}\label{eq:mass-max-moment}
		\E\left[\sup_{t\ge0}M(t)^s\right]
		\le \frac{M_0^s}{1-s},
	\end{equation}
	and
	\begin{equation}\label{eq:mass-moment-extinction}
		\lim_{t\to\infty}\E\|u(t)\|_{L^1(\D)}^s
		=\lim_{t\to\infty}\E\|v(t)\|_{L^1(\D)}^s=0.
	\end{equation}
\end{theorem}

\begin{theorem}[Fractional moment decay]
\label{thm:mass-decay-rate}
Let $(u,v)$ be a global nonnegative solution under the framework
of Section~\ref{sec:setting}, with $r>1$ and $\sigma_0>0$.
Write
\[
M(t)=\|u(t)\|_1,\qquad N(t)=\|v(t)\|_1,
\qquad M_0=M(0),\quad N_0=N(0).
\]
For $0<s<1$, set
\[
\gamma=2r-2>0,\qquad
k_s=\frac{\gamma(1-s)}2\sigma_0^2|\D|^{-\gamma}.
\]
If $M_0>0$, then, for every $t\ge0$,
\begin{equation}\label{eq:mass-decay-rate}
\E M(t)^s
\le\left(M_0^{-\gamma}+k_st\right)^{-s/\gamma}
\end{equation}
and
\begin{equation}\label{eq:future-mass-supremum}
\E\sup_{\ell\ge t}M(\ell)^s
\le\frac1{1-s}
\left(M_0^{-\gamma}+k_st\right)^{-s/\gamma}.
\end{equation}
For $\tau=0$, the bound \eqref{eq:mass-decay-rate} also holds
with $M(t)$ replaced by $N(t)$. For $\tau>0$,
\begin{equation}\label{eq:signal-decay-rate}
\E N(t)^s
\le e^{-st/(2\tau)}
\left(N_0^s+\frac{M_0^s}{1-s}\right)
+\frac1{1-s}
\left(M_0^{-\gamma}+\frac{k_st}{2}\right)^{-s/\gamma}.
\end{equation}
Consequently,
\[
\E\|u(t)\|_1^s+\E\|v(t)\|_1^s
\le C_s(1+t)^{-s/(2r-2)},\qquad t\ge0,
\]
where $C_s$ may depend on the fixed parameters and initial data.
If $M_0=0$, then $M\equiv0$; moreover, $N\equiv0$ for $\tau=0$
and $N(t)=e^{-t/\tau}N_0$ for $\tau>0$.
\end{theorem}

\section{Local well-posedness and preliminary estimates}\label{sec:local}
Throughout this section, set
\[
\alpha=\frac14,\qquad C=C(\overline\D),\qquad
z_+=\max\{z,0\},\qquad z_-=\max\{-z,0\}.
\]
All constants may depend on $\D,d,q,\chi,\tau,\sigma_0,r$ when
the corresponding parameters are present. A subscript $T$ allows
dependence on the time horizon, on $\|u_0\|_\infty$, and, when
$\tau>0$, on $\|v_0\|_{C^2}$. 

\subsection{Neumann semigroup estimates}
For $0<t\le T$, the estimates needed in the argument are
\begin{align}
	\|\nabla S_\tau(t)f\|_\infty
	&\le C_Tt^{-1/2}\|f\|_\infty,\label{eq:heat-grad}\\
	\|S(t)\diver F\|_{C^\alpha}
	&\le C_Tt^{-(1+\alpha)/2}\|F\|_\infty,\label{eq:heat-div}\\
	\|BS_\tau(t)f\|_\infty
	&\le C_Tt^{-1}\|f\|_\infty,\label{eq:heat-B0}\\
	\|BS_\tau(t)f\|_\infty
	&\le C_Tt^{-1+\alpha/2}\|f\|_{C^\alpha}.\label{eq:heat-Balpha}
\end{align}
The $S_\tau$ estimates are used only for $\tau>0$; see \cite{Lunardi,WinklerSmall} for analytic-semigroup and Neumann heat-kernel estimates. The weak Neumann divergence is defined by
\begin{equation}\label{eq:local-weak-divergence}
	\langle\diver F,\psi\rangle=-\int_\D F\cdot\nabla\psi\dd x,\qquad \psi\in H^1(\D).
\end{equation}
Writing $H(t,x,y)$ for the Neumann heat kernel, differentiated Gaussian bounds give
\[
\|S(t)\diver F\|_\infty\le C_Tt^{-1/2}\|F\|_\infty,
\qquad
\|S(t)\diver F\|_{W^{1,\infty}}\le C_Tt^{-1}\|F\|_\infty.
\]
Interpolation proves \eqref{eq:heat-div}. For \eqref{eq:heat-Balpha}, kernel cancellation yields
\[
|\Delta S(t)f(x)|
=\left|\int_\D\Delta_xH(t,x,y)(f(y)-f(x))\dd y\right|
\le C_Tt^{-1+\alpha/2}[f]_{C^\alpha}.
\]
Combining this with $S_\tau(t)=e^{-t/\tau}S(t/\tau)$ proves \eqref{eq:heat-Balpha}; the zeroth-order term is included in $C_T$. The gradient and analytic-semigroup bounds similarly give \eqref{eq:heat-grad} and \eqref{eq:heat-B0}.

\subsection{Stochastic convolution}
\begin{lemma}\label{lem:stoch}
	Let $g$ be predictable with values in $L^q(\D)$ and suppose
	\[
	\E\int_0^T\|g(s)\|_{L^q}^q\dd s<\infty.
	\]
	Then $Z(t)=\int_0^tS(t-s)g(s)\dd W_s$ has a version in $C([0,T];C^\alpha)$ and
	\begin{equation}\label{eq:local-stoch-general}
		\E\sup_{t\le T}\|Z(t)\|_{C^\alpha}^q
		\le C_TT^{q/4-1}\E\int_0^T\|g(s)\|_{L^q}^q\dd s.
	\end{equation}
	In particular, if $\|g(s)\|_\infty\le M$ almost surely on $[0,T]$, then
	\begin{equation}\label{eq:stoch-bound}
		\E\sup_{t\le T}\|Z(t)\|_{C^\alpha}^q\le C_TM^q.
	\end{equation}
	For $0<T\le1$, the right-hand side of \eqref{eq:stoch-bound} can be replaced by $CT^{q/4}M^q$.
\end{lemma}
The proof is given in Appendix~\ref{app:stoch}.

\begin{remark}[Stopped integrands]\label{rem:dimension}
	For a stopping time $\vartheta$, define the extended convolution
	\begin{equation}\label{eq:local-stopped-convolution}
		g^\vartheta(s)=\ind_{\{s\le\vartheta\}}g(s),\qquad
		Z^\vartheta(t)=\int_0^tS(t-s)g^\vartheta(s)\dd W_s.
	\end{equation}
	It agrees with $Z(t)$ for $t\le\vartheta$. Applying the preceding lemma to $g^\vartheta$ gives
	\begin{equation}\label{eq:local-stopped-bound}
		\E\sup_{t\le T\wedge\vartheta}\|Z(t)\|_{C^\alpha}^q
		\le C_TT^{q/4-1}\E\int_0^{T\wedge\vartheta}\|g(s)\|_{L^q}^q\dd s.
	\end{equation}
%
	
\end{remark}

\subsection{Logarithmic signal estimate}
\begin{lemma}\label{lem:log}
	Let $\tau>0$, $1<p<\infty$, $v_0\in D(B)$ and $f\in L^p(0,T;C^\alpha)$. Define
	\[
	w(t)=S_\tau(t)v_0+\frac1\tau\int_0^tS_\tau(t-s)f(s)\dd s.
	\]
	For $0<\varepsilon<\min\{1,T\}$, one has
	\begin{align}\label{eq:log-signal}
		\|w_t\|_{L^p(0,T;C)}
		&\le C_T\|Bv_0\|_\infty
		+C_T(1+|\log\varepsilon|)\|f\|_{L^p(0,T;C)}\notag\\
		&\quad+C_T\varepsilon^{\alpha/2}\|f\|_{L^p(0,T;C^\alpha)}.
	\end{align}
\end{lemma}
\begin{proof}
	We first take $f$ smooth in time with values in $C^\alpha$.
	Using $\partial_tS_\tau(t)=-BS_\tau(t)$, differentiation of
	the variation-of-constants formula gives
	\begin{equation}\label{eq:signal-derivative}
		w_t(t)
		=-S_\tau(t)Bv_0+\frac{f(t)}{\tau}
		-\frac1\tau\int_0^tBS_\tau(t-s)f(s)\dd s.
	\end{equation}
	The differentiation is justified by \eqref{eq:heat-Balpha},
	whose time singularity is integrable since $\alpha>0$.
	We obtain estimates independent of the time derivatives of $f$
	and then pass to general $f$ by approximation.
	
	Fix $0<\varepsilon<\min\{1,T\}$ and split the last integral as
	\[
	\int_0^tBS_\tau(t-s)f(s)\dd s
	=J_{<\varepsilon}(t)+J_{\ge\varepsilon}(t),
	\]
	where
	\begin{align*}
		J_{<\varepsilon}(t)
		&=\int_{(t-\varepsilon)_+}^t
		BS_\tau(t-s)f(s)\dd s,\\
		J_{\ge\varepsilon}(t)
		&=\ind_{\{t>\varepsilon\}}
		\int_0^{t-\varepsilon}BS_\tau(t-s)f(s)\dd s.
	\end{align*}
	For the long-time part, \eqref{eq:heat-B0} gives
	\[
	\|J_{\ge\varepsilon}(t)\|_\infty
	\le C_T\int_0^t
	\ind_{\{t-s\ge\varepsilon\}}(t-s)^{-1}
	\|f(s)\|_\infty\dd s.
	\]
	For the short-time part, we instead use
	\eqref{eq:heat-Balpha} to obtain
	\[
	\|J_{<\varepsilon}(t)\|_\infty
	\le C_T\int_0^t
	\ind_{\{t-s<\varepsilon\}}(t-s)^{-1+\alpha/2}
	\|f(s)\|_{C^\alpha}\dd s.
	\]
	The corresponding time kernels satisfy
	\[
	\int_\varepsilon^T h^{-1}\dd h
	=\log(T/\varepsilon)
	\le C_T(1+|\log\varepsilon|),
	\qquad
	\int_0^\varepsilon h^{-1+\alpha/2}\dd h
	=\frac2\alpha\varepsilon^{\alpha/2}.
	\]
	Young's convolution inequality therefore yields
	\begin{align}
		\|J_{\ge\varepsilon}\|_{L^p(0,T;C)}
		&\le C_T(1+|\log\varepsilon|)
		\|f\|_{L^p(0,T;C)},\notag\\
		\|J_{<\varepsilon}\|_{L^p(0,T;C)}
		&\le C_T\varepsilon^{\alpha/2}
		\|f\|_{L^p(0,T;C^\alpha)}.
		\label{eq:local-signal-split}
	\end{align}
	Combining these estimates with
	\[
	\|S_\tau(\cdot)Bv_0\|_{L^p(0,T;C)}
	\le C_T\|Bv_0\|_\infty,
	\qquad
	\|f/\tau\|_{L^p(0,T;C)}
	=\tau^{-1}\|f\|_{L^p(0,T;C)}
	\]
	proves \eqref{eq:log-signal} for smooth $f$.
	Here and below, the constants are independent of $\varepsilon$.
	
	To remove the time-smoothness assumption, choose
	$f_n\in C^\infty([0,T];C^\alpha)$ such that
	$f_n\to f$ in $L^p(0,T;C^\alpha)$, and let $w_n$ be the
	corresponding variation-of-constants solutions.
	The operator
	\[
	(\mathcal Kf)(t)
	:=\int_0^tBS_\tau(t-s)f(s)\dd s
	\]
	is bounded from $L^p(0,T;C^\alpha)$ to $L^p(0,T;C)$:
	indeed, \eqref{eq:heat-Balpha} and Young's inequality give
	\[
	\|\mathcal Kf\|_{L^p(0,T;C)}
	\le C_T\|f\|_{L^p(0,T;C^\alpha)}.
	\]
	Consequently, $w_n\to w$ in $L^p(0,T;C)$ and
	\[
	(w_n)_t\longrightarrow
	-S_\tau(\cdot)Bv_0+\frac f\tau-\frac1\tau\mathcal Kf
	\quad\text{in }L^p(0,T;C).
	\]
	Passing to the limit in the weak derivative identity shows that
	$w\in W^{1,p}(0,T;C)$ and that
	\eqref{eq:signal-derivative} holds almost everywhere in time.
	Finally, passing to the limit in the estimate for $w_n$
	proves \eqref{eq:log-signal} for the stated $f$.
\end{proof}

\subsection{Local well-posedness}
\begin{proposition}\label{prop:local}
	There is a pathwise unique maximal nonnegative local mild solution  of system~\eqref{eq:model} with lifetime $\zeta>0$. Moreover,
	\begin{equation}\label{eq:alternative}
		\{\zeta<\infty\}\subseteq
		\left\{\limsup_{t\uparrow\zeta}\|u(t)\|_\infty=\infty\right\}
		\quad\text{almost surely}.
	\end{equation}
For each bounded localizing time $\vartheta<\zeta$, almost surely,
\[
u\in L^2(0,\vartheta;H^1(\D)),\qquad
u-S(\cdot)u_0\in C([0,\vartheta];C^\alpha),
\]
where the latter process is zero at time $0$.
Moreover, $u$ has continuous $C^\alpha$-valued paths on every
compact subinterval of $(0,\zeta)$, and $\nabla v$ is bounded
before $\vartheta$. When $\tau>0$,
$v\in C^1([0,\vartheta];C)$ almost surely.
\end{proposition}
The proof is given in Appendix~\ref{app:local}.

\begin{proposition}
\label{prop:lp-continuation}
Assume the framework of Section~\ref{sec:setting}, with $r>1$,
and let $(u,v)$ be the maximal local solution from
Proposition~\ref{prop:local}, with lifetime $\zeta$.
For the fixed exponent $q>\max\{8,4d\}$, set $P=rq$.
We have
\begin{equation}\label{eq:lp-alternative}
    \{\zeta<\infty\}
    \subseteq
    \left\{
        \limsup_{t\uparrow\zeta}\|u(t)\|_{L^P}=\infty
    \right\}
    \qquad\text{almost surely}.
\end{equation}
\end{proposition}

\begin{proof}
Fix $T>0$ and $R>\max\{1,\|u_0\|_{L^P}\}$, and define the
stopping time
\[
\kappa_R=\zeta\wedge
\inf\{t<\zeta:\|u(t)\|_{L^P}\ge R\},
\]
with $\inf\varnothing=\infty$.
Since $P=rq>d$, one has
\[
\mu_P:=\frac12+\frac{d}{2P}<1.
\]
For $\tau>0$, the formula \eqref{eq:mild-v} and the
Neumann heat-kernel gradient estimate give
\[
\|\nabla v(t)\|_\infty
\le C_T+C_T\int_0^t(t-s)^{-\mu_P}\|u(s)\|_{L^P}\dd s
\le C_T(1+R),\qquad t<T\wedge\kappa_R.
\]
Since $v_0\in C^2(\overline\D)$ and $\partial_\nu v_0=0$,
the initial semigroup term satisfies
\[
\sup_{0\le t\le T}\|\nabla S_\tau(t)v_0\|_\infty
\le C_T\|v_0\|_{C^1}.
\]
For $\tau=0$, the same conclusion follows from
$\|Ku\|_{W^{2,P}}\le C\|u\|_{L^P}$ and the embedding
$W^{2,P}\hookrightarrow W^{1,\infty}$.
The constants involved may depend on $T$ and the prescribed
data, but not on $R$.

On the full deterministic interval $[0,T]$, define
\[
F_R(t)=\ind_{\{t<\kappa_R\}}u(t)\nabla v(t),
\qquad
G_R(t)=\ind_{\{t<\kappa_R\}}u(t)^r,
\]
setting them equal to zero at and after $\kappa_R$.
In particular, no value of $u(\zeta)$ is needed.
The gradient bound and $P=rq$ imply, almost everywhere on $[0,T]$,
\[
\|F_R(t)\|_{L^P}\le C_TR(1+R),
\qquad
\|G_R(t)\|_{L^q}
=\|\ind_{\{t<\kappa_R\}}u(t)\|_{L^P}^{r}
\le R^r.
\]
Lemma~\ref{lem:stoch}, applied in $L^q$, therefore gives a version
of
\[
Z_R(t)=\int_0^tS(t-s)G_R(s)\dd W_s
\]
with continuous $C^\alpha$-valued paths and
\[
\E\sup_{t\le T}\|Z_R(t)\|_{C^\alpha}^q
\le C_TR^{rq}<\infty.
\]

For the deterministic convolution, use the weak Neumann divergence
from \eqref{eq:local-weak-divergence}.
If $H$ denotes the Neumann heat kernel, then
\[
S(h)\diver F(x)
=-\int_\D\nabla_yH(h,x,y)\cdot F(y)\dd y,
\]
and the differentiated kernel bounds and H\"older's inequality imply
\[
\|S(h)\diver F\|_\infty
\le C_Th^{-\mu_P}\|F\|_{L^P},\qquad 0<h\le T.
\]
Locality of the stochastic integral identifies $Z_R$ with the
original stochastic convolution before $\kappa_R$.
Using the mild equation and $\mu_P<1$, we obtain
\[
\sup_{t<T\wedge\kappa_R}\|u(t)\|_\infty
\le \|u_0\|_\infty+C_TR(1+R)
+\sigma_0\sup_{t\le T}\|Z_R(t)\|_\infty
<\infty
\qquad\text{almost surely}.
\]
On the event
\[
\left\{\zeta\le T,\quad
\sup_{t<\zeta}\|u(t)\|_{L^P}<R\right\},
\]
one has $\kappa_R=\zeta$, so this bound contradicts
\eqref{eq:alternative}.
Taking a countable union over integer $R$ and $T$ excludes
finite lifetime with bounded $L^P$ norm.
Finally, continuity on compact subintervals of $[0,\zeta)$
shows that a finite $\limsup_{t\uparrow\zeta}\|u(t)\|_{L^P}$
would imply $\sup_{t<\zeta}\|u(t)\|_{L^P}<\infty$.
This proves \eqref{eq:lp-alternative}.
\end{proof}

\subsection{Mass control}
\begin{lemma}\label{lem:local-mass}
	Let $r>1$ and let $(u,v)$ be the maximal local solution of system~\eqref{eq:model}  with lifetime $\zeta$. Set
	\[
	M(t)=\int_\D u(t,x)\dd x,\qquad H_r(t)=\int_\D u(t,x)^r\dd x,\qquad t<\zeta.
	\]
	Then
	\begin{equation}\label{eq:local-mass-sde}
		M(t)=M_0+\sigma_0\int_0^tH_r(s)\dd W_s,
	\end{equation}
	and, almost surely,
	\begin{equation}\label{eq:local-mass-control}
		\sup_{t<\zeta}M(t)<\infty,\qquad
		\int_0^\zeta H_r(t)^2\dd t<\infty.
	\end{equation}
\end{lemma}
\begin{proof}
	For integers $n>\|u_0\|_\infty$, define the bounded localizing times
	\begin{equation}\label{eq:local-mass-stops}
		\zeta_n=n\wedge\inf\{t<\zeta:\|u(t)\|_\infty\ge n\}.
	\end{equation}
	Then,  $\zeta_n\uparrow\zeta$, with $\zeta_n<\zeta$ on $\{\zeta<\infty\}$.
	Testing the stopped weak equation  \eqref{eq:model} against the constant function $1$ proves \eqref{eq:local-mass-sde}: the Neumann diffusion and divergence terms integrate to zero.

	For $K>M_0$, set
	$\theta_K=\inf\{t<\zeta:M(t)\ge K\}$, with $\inf\varnothing=\infty$. Since $M(\theta_K)=K$ on $\{\theta_K\le T\wedge\zeta_n\}$, optional stopping gives
	\[
	K\Prob(\theta_K\le T\wedge\zeta_n)
	\le \E M(T\wedge\zeta_n\wedge\theta_K)=M_0.
	\]
	Letting $n\to\infty$ and then $T\to\infty$ yields
	$\Prob(\sup_{t<\zeta}M(t)>K)\le M_0/K$.
	Consequently, the mass is bounded on almost every path.
	It\^o's isometry also gives
	\[
	\sigma_0^2\E\int_0^{T\wedge\zeta_n\wedge\theta_K}H_r(s)^2\dd s
	=\E M(T\wedge\zeta_n\wedge\theta_K)^2-M_0^2\le K^2.
	\]
	Monotone convergence, first in $n$ and then in $T$, shows that
	$\int_0^{\zeta\wedge\theta_K}H_r(s)^2\dd s<\infty$
	almost surely for each fixed $K>M_0$. Almost every path satisfies
	$\sup_{t<\zeta}M(t)<K$ for one such integer, for which
	$\theta_K=\infty$. This proves \eqref{eq:local-mass-control}.
\end{proof}

\section{Scalar comparison}\label{sec:comparison}
\begin{proposition}[Scalar supersolution]\label{lem:comparison}
	Assume either $r>3/2$, or $\tau=0$, $r=3/2$, and $\sigma_0^2\ge2\chi$.
	Let $(u,v)$ be the maximal local solution with lifetime $\zeta$.
	Set $a=0$ if $\tau=0$, and, if $\tau>0$, set
	\begin{equation}\label{eq:a}
		\eta(t)=\|(v_t(t))_-\|_\infty,
		\qquad a(t)=\chi\tau\eta(t),\qquad t<\zeta,
	\end{equation}
	where $z_-:=\max\{-z,0\}$. The scalar equation
	\begin{equation}\label{eq:scalar}
		\mathrm dU=(\chi U^2+a(t)U)\,\mathrm dt+\sigma_0U^r\,\mathrm dW_t,
		\qquad U(0)=U_0:=1+\|u_0\|_\infty,
	\end{equation}
	has a nonnegative solution on every compact subinterval of $[0,\zeta)$, and
	\begin{equation}\label{eq:domination}
		0\le u(t,x)\le U(t),\qquad t<\zeta,\quad x\in\overline\D.
	\end{equation}
	When $a=0$, $U$ is defined on $[0,\infty)$ almost surely.
\end{proposition}
\begin{proof}
	The proof consists of a scalar nonexplosion estimate and a positive-part energy estimate.
	
	\emph{Step 1: scalar nonexplosion.}
Proposition~\ref{prop:local} implies that $a$ is adapted and bounded on compact time intervals before $\zeta$, pathwise. Therefore, the coefficients $x\mapsto\chi x^2+a(t)x$ and $x\mapsto\sigma_0(x_+)^r$ are locally Lipschitz in $x$ after localization of $a$.
Equation~\eqref{eq:scalar} therefore has a unique local solution $U$. Since the coefficients vanish at zero, the comparison theorem for scalar SDEs gives $U\ge0$.

	Assume first that $r>3/2$. For $0<\lambda<1$, let $V(x)=(1+x)^\lambda$.
	The generator applied to $V$ is 
	\begin{equation}\label{eq:generator}
		\mathcal L_aV(x)
		=\lambda(1+x)^{\lambda-1}(\chi x^2+ax)
		-\frac{\sigma_0^2\lambda(1-\lambda)}2
		(1+x)^{\lambda-2}x^{2r}.
	\end{equation}
	For $x\ge1$ and $0\le a\le M$,
	\begin{equation}\label{eq:detail-scalar-orders}
		\mathcal L_aV(x)
		\le\lambda\chi x^{\lambda+1}
		+\lambda Mx^\lambda-c_\lambda x^{\lambda+2r-2},
		\qquad
		c_\lambda=2^{\lambda-3}\sigma_0^2\lambda(1-\lambda)>0.
	\end{equation}
	Since $2r-2>1$, the right-hand side is bounded above, uniformly for $a\in[0,M]$. Thus
	\[
	\sup_{x\ge0,\,0\le a\le M}\mathcal L_aV(x)\le C_M.
	\]
	Let \(\xi\) denote the maximal lifetime of the scalar solution.
	For \(L>U_0\) and \(M>0\), define
	\[
	\widehat\rho_L=\inf\{t\in[0,\xi):U(t)\ge L\},
	\qquad
	\kappa_M=\inf\{t\in[0,\zeta):a(t)>M\},
	\]
	with \(\inf\varnothing=\infty\).
	Fix an interior localizing time \(\zeta_m\) from
	\eqref{eq:local-mass-stops} and set
	\[
	\sigma_{L,M,m}=T\wedge\zeta_m\wedge\widehat\rho_L\wedge\kappa_M.
	\]
	The scalar continuation criterion ensures that these stopped values
	are defined. It\^o's formula gives
	\begin{align}\label{eq:detail-scalar-local-exit}
		\E V(U(\sigma_{L,M,m}))
		&=V(U_0)+\E\int_0^{\sigma_{L,M,m}}\mathcal L_aV(U(t))\dd t
		\le V(U_0)+C_MT,\notag\\
		\Prob\left(\widehat\rho_L\le T\wedge\zeta_m\wedge\kappa_M\right)
		&\le\frac{V(U_0)+C_MT}{(1+L)^\lambda}
		\longrightarrow0
		\qquad(L\to\infty).\notag
	\end{align}
	For fixed \(m\), the coefficient \(a\) is bounded on
	\([0,\zeta_m]\) almost surely. Letting first \(L\to\infty\)
	and then \(M\to\infty\) therefore excludes scalar explosion
	before \(T\wedge\zeta_m\). Finally, letting \(m\to\infty\)
	proves existence on every compact subinterval of \([0,\zeta)\).
	If \(a=0\), the scalar equation is defined for all times, and
	the same estimate without \(\zeta_m\) or \(\kappa_M\) proves
	nonexplosion on every deterministic finite interval.

	For $r=3/2$, $a=0$, and $\sigma_0^2\ge2\chi$, use $V_0(x)=\log(1+x)$ instead. Direct differentiation gives
	\begin{equation}\label{eq:critical-generator}
		\mathcal L_0V_0(x)
		=\frac{\chi x^2}{1+x}-\frac{\sigma_0^2x^3}{2(1+x)^2}
		=\frac{x^2}{(1+x)^2}
		\left[\chi+\left(\chi-\frac{\sigma_0^2}{2}\right)x\right]
		\le\chi.
	\end{equation}
	The resulting bound is
	\[
	\Prob\left(\widehat\rho_L\le T\right)
	\le\frac{\log(1+U_0)+\chi T}{\log(1+L)}\longrightarrow0,
	\]
	which proves scalar nonexplosion also in this case.
	
	\emph{Step 2: comparison.}
	The signal equation and $u,v\ge0$ imply
	\begin{equation}\label{eq:laplace-bound}
		-\chi\Delta v
		=\begin{cases}
			\chi u-\chi v,&\tau=0,\\
			\chi u-\chi v-\chi\tau v_t,&\tau>0,
		\end{cases}
		\qquad -\chi\Delta v\le\chi u+a(t).
	\end{equation}
	Define the nonpositive remainder
	\[
	\mathcal R(t,x):=u(t,x)\bigl[-\chi\Delta v(t,x)-\chi u(t,x)-a(t)\bigr]\le0.
	\]
	For $w=u-U$ and $z=w_+$, the difference equation is
	\begin{align*}
		\mathrm dw={}&\bigl[\Delta w-\chi\nabla v\cdot\nabla w
		+\chi(u^2-U^2)+a(t)w+\mathcal R\bigr]\dd t
		+\sigma_0(u^r-U^r)\dd W_t.
	\end{align*}
	Choose auxiliary stopping times $\vartheta_K$ increasing to the joint lifetime so that $u,U,a,\nabla v$, and $\Delta v$ are bounded by $K$ before $\vartheta_K$.
	The regularity in Proposition~\ref{prop:local} allows us to apply It\^o's formula by smooth convex approximation. Before $\vartheta_K$, we obtain
	\begin{align}\label{eq:detail-positive-ito}
		\mathrm d\|z\|_2^2
		={}&-2\|\nabla z\|_2^2\dd t
		-2\chi\int_\D z\nabla v\cdot\nabla z\dd x\dd t\notag\\
		&+2\chi\int_\D z(u^2-U^2)\dd x\dd t
		+2a(t)\|z\|_2^2\dd t+2\int_\D z\mathcal R\dd x\dd t\notag\\
		&+\sigma_0^2\int_\D\ind_{\{u>U\}}(u^r-U^r)^2\dd x\dd t
		+2\sigma_0\left(\int_\D z(u^r-U^r)\dd x\right)\dd W_t.
	\end{align}
	For the terms on the right-hand side, we have
	\begin{align*}
		-2\chi\int_\D z\nabla v\cdot\nabla z\dd x
		&=\chi\int_\D\Delta v\,z^2\dd x
		\le\chi K\|z\|_2^2,\\
		2\chi\int_\D z(u^2-U^2)\dd x
		&=2\chi\int_\D(u+U)z^2\dd x\le4\chi K\|z\|_2^2,\\
		2a(t)\|z\|_2^2+2\int_\D z\mathcal R\dd x
		&\le2K\|z\|_2^2,\\
		\sigma_0^2\int_\D\ind_{\{u>U\}}(u^r-U^r)^2\dd x
		&\le\sigma_0^2r^2K^{2r-2}\|z\|_2^2.
	\end{align*}
	The first line uses $\partial_\nu v=0$; the last uses the mean value theorem on $[0,K]$. Hence it follows from \eqref{eq:detail-positive-ito} that
	\begin{equation}\label{eq:positive-comparison}
		\mathrm d\|z\|_2^2+2\|\nabla z\|_2^2\dd t
		\le C_K\|z\|_2^2\dd t
		+2\sigma_0\left(\int_\D z(u^r-U^r)\dd x\right)\dd W_t.
	\end{equation}
Since $z(0)=0$, taking expectation yields
	\[
	0\le\E\|z(t\wedge\vartheta_K)\|_2^2
	\le C_K\int_0^t\E\|z(s\wedge\vartheta_K)\|_2^2\dd s=0
	\]
	by Gronwall's inequality. Removing $\vartheta_K$ and using spatial and temporal continuity proves~\eqref{eq:domination}.
\end{proof}

\section{The parabolic--elliptic system}\label{sec:parabolic-elliptic}
\begin{proof}[Proof of Theorem~\ref{thm:global} for $\tau=0$]
	For $a=0$ and $0<\lambda<1$, \eqref{eq:generator} gives $\mathcal L_0V\le C$.
	Writing $\rho_L=\inf\{t\ge0:U(t)\ge L\}$, we obtain
	\begin{align}\label{eq:detail-pe-exit}
		(1+L)^\lambda\Prob\left(\rho_L\le T\right)
		&\le\E(1+U(T\wedge\rho_L))^\lambda\notag\\
		&\le(1+U_0)^\lambda+CT.
	\end{align}
	The scalar process is global by Proposition~\ref{lem:comparison}. On the event $\{\zeta\le T\}$,
	\[
	\limsup_{t\uparrow\zeta}\|u(t)\|_\infty=\infty,
	\qquad
	\sup_{0\le t<\zeta}\|u(t)\|_\infty
	\le\sup_{0\le t\le T}U(t)<\infty,
	\]
	which is impossible. Thus $\Prob\left(\zeta\le T\right)=0$ for each $T<\infty$.
	After continuation,
	\[
	\left\{\sup_{0\le t\le T}\|u(t)\|_\infty\ge R\right\}
	\subseteq\left\{\sup_{0\le t\le T}U(t)\ge R\right\}
	=\{\rho_R\le T\}.
	\]
	Estimate~\eqref{eq:detail-pe-exit} proves \eqref{eq:pe-tail}. Nonnegativity and pathwise uniqueness follow from Proposition~\ref{prop:local}.
\end{proof}

\begin{proof}[Proof of Theorem~\ref{thm:critical}\ref{item:critical-pe}]
	For $V_0(x)=\log(1+x)$, \eqref{eq:critical-generator} yields
	\begin{align*}
		\log(1+L)\Prob\left(\rho_L\le T\right)
		&\le\E\log(1+U(T\wedge\rho_L))\\
		&\le\log(1+U_0)+\chi T.
	\end{align*}
	Letting $L\to\infty$ proves nonexplosion of $U$. Comparison and the continuation criterion then give $\zeta=\infty$ almost surely, exactly as in the preceding proof. With $L=R>U_0$, the same inequalities give \eqref{eq:critical-pe-tail}.
\end{proof}

\section{The fully parabolic system}\label{sec:global}
\subsection{Stopped estimates}\label{sec:stopping}
Assume throughout this section that $\tau>0$ and $r>3/2$.
For the scalar process in \eqref{eq:scalar}, define
\begin{equation}\label{eq:rho}
	\rho_L=\inf\{t\in[0,\zeta):U(t)\ge L\},
	\qquad L>U_0,\qquad\inf\varnothing=\infty.
\end{equation}
The comparison \eqref{eq:domination} implies that
\begin{equation}\label{eq:detail-stopping-validity}
	\zeta<\infty
	\quad\Longrightarrow\quad
	\limsup_{t\uparrow\zeta}U(t)=\infty
	\quad\Longrightarrow\quad
	\rho_L<\zeta\quad\text{for every }L>U_0.
\end{equation}
Observe that, for $t<\rho_L$,
\begin{equation}\label{eq:detail-stopped-amplitude}
	0\le u(t,x)\le U(t)<L.
\end{equation}

\begin{lemma}[Spatial regularity]\label{lem:holder}
	Let $\alpha=1/4$, $\beta=\max\{2,r\}$, and $1<p<2$. Then
	\begin{equation}\label{eq:holder}
		\E\int_0^{T\wedge\rho_L}\|u(t)\|_{C^\alpha}^p\dd t
		\le C_T(1+L^{\beta p}).
	\end{equation}
\end{lemma}
\begin{proof}
	Let $B_C$ be the realization of $B$ on $C$.
	The Neumann resolvent estimate and $v_0\in C^2(\overline\D)$,
	$\partial_\nu v_0=0$, give $v_0\in D(B_C)$ and
	\[
	\sup_{t\le T}\|\nabla S_\tau(t)v_0\|_\infty
	\le C_T\bigl(\|v_0\|_\infty+\|B_Cv_0\|_\infty\bigr).
	\]
	Thus, for $t<T\wedge\rho_L$, the variation-of-constants formula gives
	\begin{align*}
	\|\nabla v(t)\|_\infty
	&\le C_T+C_TL\int_0^t(t-s)^{-1/2}\dd s
	\le C_T(1+L).
	\end{align*}
	Consequently,
	\begin{equation}\label{eq:gradient-L}
	\sup_{t<T\wedge\rho_L}\|\nabla v(t)\|_\infty\le C_T(1+L),
	\qquad
	\ind_{t<\rho_L}\|u(t)\nabla v(t)\|_\infty\le C_T(1+L^2).
	\end{equation}
	Set, on $[0,T]$,
	\begin{align}\label{eq:detail-extended-mild}
	\widetilde u_L(t)={}&S(t)u_0+D_L(t)+Z_L(t),\notag\\
	D_L(t):={}&-\chi\int_0^tS(t-s)\diver\bigl(\ind_{s<\rho_L}u(s)\nabla v(s)\bigr)\dd s,\notag\\
	Z_L(t):={}&\sigma_0\int_0^tS(t-s)\ind_{s<\rho_L}u(s)^r\dd W_s.
	\end{align}
	For $t<\rho_L$, one has $\widetilde u_L(t)=u(t)$.
	For $0<t\le T$, heat-semigroup smoothing gives
	$\|S(t)u_0\|_{C^\alpha}\le C_Tt^{-\alpha/2}\|u_0\|_\infty$.
	Since $\alpha=1/4$ and $1<p<2$, it follows that
	\[
	\int_0^T\|S(t)u_0\|_{C^\alpha}^p\dd t
	\le C_T\|u_0\|_\infty^p
	\frac{T^{1-\alpha p/2}}{1-\alpha p/2}<\infty.
	\]
	The deterministic convolution satisfies
	\begin{align*}
	\sup_{t\le T}\|D_L(t)\|_{C^\alpha}
	&\le C_T(1+L^2)\sup_{t\le T}
	\int_0^t(t-s)^{-(1+\alpha)/2}\dd s\\
	&\le C_T(1+L^2).
	\end{align*}
	For the stochastic convolution, $p<q$ and Lemma~\ref{lem:stoch} give
	\[
	\E\sup_{t\le T}\|Z_L(t)\|_{C^\alpha}^p
	\le\left(\E\sup_{t\le T}\|Z_L(t)\|_{C^\alpha}^q\right)^{p/q}
	\le C_TL^{rp}.
	\]
	Combining these estimates yields
	\begin{align*}
	\E\int_0^{T\wedge\rho_L}\|u(t)\|_{C^\alpha}^p\dd t
	&\le C_p\int_0^T\|S(t)u_0\|_{C^\alpha}^p\dd t\\
	&\quad+C_pT\E\sup_{t\le T}\|D_L(t)\|_{C^\alpha}^p+C_pT\E\sup_{t\le T}\|Z_L(t)\|_{C^\alpha}^p\\
	&\le C_T(1+L^{2p}+L^{rp})
	\le C_T(1+L^{\beta p}).
	\end{align*}
\end{proof}
Recall the random reaction coefficient $a$ defined in \eqref{eq:a}. We have the following estimate.
\begin{lemma}[Logarithmic coefficient bound]\label{lem:a}
	For every $1<p<2$ and all sufficiently large $L$,
	\begin{equation}\label{eq:a-log}
		\E\int_0^{T\wedge\rho_L}a(t)^p\dd t
		\le C_T+C_T(1+\log L)^p
		\E\int_0^{T\wedge\rho_L}U(t)^p\dd t.
	\end{equation}
\end{lemma}
\begin{proof}
	Let $f_L(t)=\ind_{t<\rho_L}u(t)$ and define
	\[
	v_L(t)=S_\tau(t)v_0+\frac1\tau\int_0^tS_\tau(t-s)f_L(s)\dd s.
	\]
	For $t<\rho_L$, $v_L(t)=v(t)$ and $\partial_tv_L(t)=v_t(t)$. On the full deterministic interval, formula~\eqref{eq:signal-derivative} holds for almost every $t$:
	\[
	\partial_tv_L(t)=-S_\tau(t)Bv_0+\frac{f_L(t)}\tau
	-\frac1\tau\int_0^tBS_\tau(t-s)f_L(s)\dd s.
	\]
	Applying Lemma~\ref{lem:log} pathwise, raising its estimate to the power $p$, and using scalar domination and Lemma~\ref{lem:holder} give
	\begin{align}\label{eq:detail-a-epsilon}
		\E\int_0^{T\wedge\rho_L}a(t)^p\dd t
		\le{}&C_T+C_T(1+|\log\varepsilon|)^p
		\E\int_0^{T\wedge\rho_L}U(t)^p\dd t\notag\\
		&+C_T\varepsilon^{\alpha p/2}(1+L^{\beta p}).
	\end{align}
	Choose
	\begin{equation}\label{eq:detail-epsilon-choice}
		\varepsilon=c_TL^{-\kappa},\qquad
		\kappa>2\beta/\alpha,\qquad0<c_T<\min\{1,T\}.
	\end{equation}

	For $L\ge1$,
	\begin{align*}
		\varepsilon^{\alpha p/2}(1+L^{\beta p})
		&=c_T^{\alpha p/2}
		\bigl[L^{-\kappa\alpha p/2}+L^{p(\beta-\kappa\alpha/2)}\bigr]
		\le C_T,\\
		1+|\log\varepsilon|
		&\le1+|\log c_T|+\kappa\log L
		\le C_T(1+\log L).
	\end{align*}
	Substitution into \eqref{eq:detail-a-epsilon} proves \eqref{eq:a-log}.
\end{proof}

\subsection{Scalar energy}
\begin{lemma}\label{lem:lyapunov}
	Let $\lambda,\gamma,n,p$ be as in \eqref{eq:exponents}. There exist $c,C>0$, independent of $a\ge0$ and $x\ge0$, such that
	\begin{equation}\label{eq:lyapunov}
		\mathcal L_a(1+x)^\lambda\le C-cx^n+Ca^p.
	\end{equation}
	Consequently,
	\begin{equation}\label{eq:stopped-energy}
		\E(1+U(T\wedge\rho_L))^\lambda
		+c\E\int_0^{T\wedge\rho_L}U(t)^n\dd t
		\le C_T+C\E\int_0^{T\wedge\rho_L}a(t)^p\dd t.
	\end{equation}
\end{lemma}
\begin{proof}
	Recall the exponent identities
	\begin{equation}\label{eq:detail-first-young-exponents}
		n=\lambda+\gamma>\lambda+1,\qquad n-\lambda=\gamma,
		\qquad \frac{\lambda}{n}+\frac\gamma n=1,
		\qquad p=\frac n\gamma.
	\end{equation}
	For $x\ge1$, equation~\eqref{eq:generator} gives
	\[
	\mathcal L_aV(x)
	\le\lambda\chi x^{\lambda+1}+\lambda ax^\lambda-c_\lambda x^n.
	\]
	The two positive terms obey, for every $\delta>0$,
	\begin{align*}
		\lambda\chi x^{\lambda+1}
		&\le\delta x^n+C_\delta,\\
		\lambda ax^\lambda
		&\le\delta(x^\lambda)^{n/\lambda}+C_\delta a^{n/(n-\lambda)}
		=\delta x^n+C_\delta a^p.
	\end{align*}
	The second inequality is Young's inequality with conjugate exponents $n/\lambda$ and $n/\gamma$. For $0\le x\le1$,
	\[
	\mathcal L_aV(x)\le C(1+a)\le C(1+a^p).
	\]
	Choosing $\delta=c_\lambda/4$ and enlarging the constant on $[0,1]$ proves \eqref{eq:lyapunov}.
	
	For fixed $L$, define $\overline U_L(t)=U(t\wedge\rho_L)$. By \eqref{eq:detail-stopping-validity}, it is well-defined on $[0,T]$ and satisfies
	\begin{align*}
		\overline U_L(t)=U_0
		&+\int_0^t\ind_{s<\rho_L}\bigl[\chi U(s)^2+a(s)U(s)\bigr]\dd s+\sigma_0\int_0^t\ind_{s<\rho_L}U(s)^r\dd W_s.
	\end{align*}
It\^o's formula and \eqref{eq:lyapunov} give
	\[
	\E V(\overline U_L(T))
	\le V(U_0)+CT-c\E\int_0^{T\wedge\rho_L}U(s)^n\dd s
	+C\E\int_0^{T\wedge\rho_L}a(s)^p\dd s,
	\]
	which is \eqref{eq:stopped-energy}.
\end{proof}

\subsection{Global existence and tail bounds}
\begin{proof}[Proof of Theorem~\ref{thm:global} for $\tau>0$]
	For $\gamma>1$ and $0<\lambda<1$,
	\begin{equation}\label{eq:detail-absorption-gap}
		1<p=1+\frac\lambda\gamma<2,
		\qquad n-p=n\left(1-\frac1\gamma\right)>0.
	\end{equation}
	Thus,
	Lemma~\ref{lem:a} provides the estimate for the $a^p$-term
	arising from Young's inequality in the proof of
	Lemma~\ref{lem:lyapunov}. Combining it with \eqref{eq:stopped-energy} gives
	\begin{align}\label{eq:before-absorb}
		\E(1+U(T\wedge\rho_L))^\lambda
		+c\E\int_0^{T\wedge\rho_L}U(t)^n\dd t
		\le C_T+C_T(1+\log L)^p
		\E\int_0^{T\wedge\rho_L}U(t)^p\dd t.
	\end{align}
	Young's inequality with conjugate exponents $n/p$ and $n/(n-p)$ gives
	\begin{equation}\label{eq:young-log}
		C_T(1+\log L)^p x^p
		\le\frac c2(x^p)^{n/p}
		+C_T\bigl[(1+\log L)^p\bigr]^{n/(n-p)}
		=\frac c2x^n+C_T(1+\log L)^{pn/(n-p)}.
	\end{equation}
Here $pn/(n-p)=n/(\gamma-1)=\nu_*$.

	Combining \eqref{eq:before-absorb} and \eqref{eq:young-log} gives
	\begin{equation}\label{eq:energy-final}
		\E(1+U(T\wedge\rho_L))^\lambda
		+\frac c2\E\int_0^{T\wedge\rho_L}U(t)^n\dd t
		\le C_T\bigl[1+(1+\log L)^{\nu_*}\bigr].
	\end{equation}
	Since $U(\rho_L)=L$ on $\{\rho_L\le T\}$,
	\begin{equation}\label{eq:rho-tail}
		\Prob\left(\rho_L\le T\right)
		\le\frac{\E(1+U(T\wedge\rho_L))^\lambda}{(1+L)^\lambda}
		\le C_T\frac{1+(1+\log L)^{\nu_*}}{(1+L)^\lambda}
		\longrightarrow0,
	\end{equation}
as $L\rightarrow \infty$.
	To pass from this estimate to global existence, take any increasing deterministic sequence $L_j\to\infty$ with $L_j>U_0$. Relation~\eqref{eq:detail-stopping-validity} implies
	\begin{equation}\label{eq:detail-global-event}
		\{\zeta\le T\}
		\subseteq\bigcap_{j\ge1}\{\rho_{L_j}<\zeta\le T\}
		\subseteq\bigcap_{j\ge1}\{\rho_{L_j}\le T\}.
	\end{equation}
	Therefore
	\[
	\Prob\left(\zeta\le T\right)\le\inf_j\Prob\left(\rho_{L_j}\le T\right)=0,
	\qquad
	\Prob\left(\zeta<\infty\right)\le\sum_{m=1}^\infty\Prob\left(\zeta\le m\right)=0.
	\]
Finally,
	\[
	\Prob\left(\sup_{t\le T}\|u(t)\|_\infty\ge R\right)
	\le\Prob\left(\rho_R\le T\right),
	\]
	and \eqref{eq:rho-tail} proves \eqref{eq:pp-tail}. The regularity and pathwise uniqueness follow from Proposition~\ref{prop:local}.
\end{proof}

\section{The two-dimensional case \texorpdfstring{$r=3/2$}{r=3/2}}\label{sec:critical-2d}
Throughout this section, let $d=2$, $r=3/2$, and $\sigma_0>0$.
 We use the local solution from Proposition~\ref{prop:local} and the stopping times \eqref{eq:local-mass-stops}. If $u_0\equiv0$, uniqueness gives $u\equiv0$ and a global linear signal, so we may exclude this case.

Lemma~\ref{lem:local-mass}, applied with $r=3/2$, gives the estimate
\begin{equation}\label{eq:critical-mass-bound}
	\int_0^\zeta\left(\int_\D u(t,x)^{3/2}\dd x\right)^2\dd t
	=\int_0^\zeta\|u(t)\|_{3/2}^3\dd t<\infty
	\quad\text{almost surely}.
\end{equation}

\subsection{Signal and energy estimates}
Fix $T<\infty$. If $\tau>0$, extend the source by zero after the lifetime:
\[
f_T(t)=\ind_{\{t<T\wedge\zeta\}}u(t),\qquad 0<t<T.
\]
Lemma~\ref{lem:local-mass} gives $f_T\in L^3(0,T;L^{3/2}(\D))$ almost surely.
Define
\begin{equation}\label{eq:mild-v-1}
		\widetilde v(t)=S_\tau(t)v_0+\frac1\tau\int_0^t S_\tau(t-s)f_T(s)\dd s
	\end{equation}

Set $X=L^{3/2}(\D)$ and let $B_X$ be the Neumann realization of
$(I-\Delta)/\tau$ on $X$. Since
$v_0\in C^2(\overline\D)$ and $\partial_\nu v_0=0$,
one has $v_0\in D(B_X)$.
Deterministic maximal regularity, applied separately to each path,
therefore yields
\begin{equation}\label{eq:critical-maximal-regularity}
\|\widetilde v\|_{L^3(0,T;W^{2,3/2})}
\le C_{T,\tau}\left(
\|v_0\|_{W^{2,3/2}}+
\|f_T\|_{L^3(0,T;L^{3/2})}\right).
\end{equation}
Here $\|v_0\|_{W^{2,3/2}}\le C_\D\|v_0\|_{C^2}$.
We refer the readers to the reference \cite{Denk} for details.
Note that $\widetilde v=v$ on $[0,T\wedge\zeta)$. Since $d=2$,
\[
W^{2,3/2}(\D)\hookrightarrow W^{1,6}(\D),
\]
and therefore
\begin{equation}\label{eq:critical-signal-integral}
	\int_0^{T\wedge\zeta}\|\nabla v(t)\|_6^3\dd t<\infty
	\qquad\text{almost surely}.
\end{equation}
When $\tau=0$, the elliptic estimate $\|\nabla Ku\|_6\le C\|u\|_{3/2}$ gives \eqref{eq:critical-signal-integral} directly.

\begin{lemma}[Localized energy estimate]\label{lem:critical-energy}
	For every fixed $p\ge2$, set
	\[
	Y_p=\int_\D u^p\dd x,\qquad
	w=u^{p/2},\qquad
	D_p=\int_\D|\nabla w|^2\dd x.
	\]
	On the stochastic interval $[0, \zeta)$, there are constants $c_p,C_p>0$ such that
	\begin{equation}\label{eq:critical-energy-bound}
		\mathrm dY_p+c_pD_p\dd t
		\le C_p\Gamma(t)Y_p\dd t+\mathrm dN_p(t),
		\qquad
		\Gamma(t)=1+\|u(t)\|_{3/2}^3+\|\nabla v(t)\|_6^3,
	\end{equation}
	where
	\[
	N_p(t)=\sigma_0p\int_0^t\left(\int_\D u(s,x)^{p+1/2}\dd x\right)\dd W_s.
	\]
	For every finite $T$, $\int_0^{T\wedge\zeta}\Gamma(t)\dd t<\infty$ almost surely.
\end{lemma}
\begin{proof}
	Before $\zeta_n$, the variational regularity from Proposition~\ref{prop:local} justifies the use of It\^o's formula for $\int u^p$.
The resulting identity is
	\begin{align}\label{eq:critical-energy-identity}
		\mathrm dY_p+\frac{4(p-1)}pD_p\dd t
		={}&2\chi(p-1)\int_\D w\nabla w\cdot\nabla v\dd x\dd t\notag\\
		&+\frac{\sigma_0^2p(p-1)}2\int_\D u w^2\dd x\dd t
		+\mathrm dN_p(t).
	\end{align}
	In two dimensions, the Gagliardo--Nirenberg inequalities give
	\begin{align*}
		\|w\|_6^2&\le C\|\nabla w\|_2^{4/3}\|w\|_2^{2/3}+C\|w\|_2^2,\\
		\|w\|_3^2&\le C\|\nabla w\|_2^{2/3}\|w\|_2^{4/3}+C\|w\|_2^2.
	\end{align*}
	Consequently, for every $\varepsilon>0$,
	\begin{align}
		\int_\D u w^2\dd x
		&\le\|u\|_{3/2}\|w\|_6^2\le\varepsilon D_p+
		C_\varepsilon\bigl(\|u\|_{3/2}^3+\|u\|_{3/2}\bigr)Y_p,
		\label{eq:critical-noise-interpolation}\\
		\int_\D|\nabla v|^2 w^2\dd x
		&\le\|\nabla v\|_6^2\|w\|_3^2\le\varepsilon D_p+
		C_\varepsilon\bigl(\|\nabla v\|_6^3+\|\nabla v\|_6^2\bigr)Y_p.
		\label{eq:critical-drift-interpolation}
	\end{align}
	Estimate the first term on the right-hand side of \eqref{eq:critical-energy-identity} by a small multiple of $D_p$ plus a constant times $\int_\D|\nabla v|^2w^2\dd x$. Then apply \eqref{eq:critical-noise-interpolation}--\eqref{eq:critical-drift-interpolation}, choosing their small constants to absorb the remaining gradient terms. The lower powers are bounded by $C\Gamma(t)$. This proves \eqref{eq:critical-energy-bound}; the time-integrability of $\Gamma$ follows from \eqref{eq:critical-mass-bound} and \eqref{eq:critical-signal-integral}.
\end{proof}

\begin{lemma}[$L^p$ bounds before explosion]\label{lem:critical-lp}
	For every fixed $p\ge2$ and $T<\infty$,
	\begin{equation}\label{eq:critical-lp-bound}
		\sup_{t<T\wedge\zeta}\|u(t)\|_p<\infty
		\qquad\text{almost surely}.
	\end{equation}
\end{lemma}
\begin{proof}
	For the stopping times \(\zeta_n\) in \eqref{eq:local-mass-stops}, define on \([0,T]\)
	\[
	Z_{p,n}(t)=
	\exp\left(-C_p\int_0^{t\wedge\zeta_n}\Gamma(s)\dd s\right)
	Y_p(t\wedge\zeta_n).
	\]
	By \eqref{eq:critical-energy-bound}, \(Z_{p,n}\) is a nonnegative
	local supermartingale and hence a supermartingale. Its maximal
	inequality gives, for \(K,R>0\),
	\[
	\Prob\left(
	\sup_{t<T\wedge\zeta_n}Y_p(t)>R,
	\ \int_0^{T\wedge\zeta}\Gamma(s)\dd s\le K
	\right)
	\le
	\Prob\left(\sup_{t\le T}Z_{p,n}(t)>Re^{-C_pK}\right)
	\le\frac{e^{C_pK}Y_p(0)}{R}.
	\]
	Letting \(n\to\infty\) yields
	\begin{equation}\label{eq:critical-lp-probability}
		\Prob\left(
		\sup_{t<T\wedge\zeta}Y_p(t)>R,
		\ \int_0^{T\wedge\zeta}\Gamma(s)\dd s\le K
		\right)
		\le\frac{e^{C_pK}Y_p(0)}{R}.
	\end{equation}
	First let \(R\to\infty\) for each fixed positive integer \(K\).
	Taking the union over these integers and using
	\(\int_0^{T\wedge\zeta}\Gamma(s)\dd s<\infty\) almost surely
	proves \eqref{eq:critical-lp-bound}.
\end{proof}

\subsection{Continuation}
\begin{proof}[Proof of Theorem~\ref{thm:critical}\ref{item:critical-2d}]
For the fixed exponent $q>8$ from Section~\ref{sec:setting},
set $P=3q/2>12$.
For every $T<\infty$, Lemma~\ref{lem:critical-lp} gives
\[
\sup_{t<T\wedge\zeta}\|u(t)\|_{L^P}<\infty
\qquad\text{almost surely}.
\]
Since $r=3/2$, this exponent equals $rq$. Proposition~\ref{prop:lp-continuation} gives
$\Prob(\zeta\le T)=0$ for every $T>0$.
Hence $\zeta=\infty$ almost surely.
Pathwise uniqueness and nonnegativity follow from
Proposition~\ref{prop:local}.
For $\tau>0$, the same proposition gives
$v\in C^1([0,T];C(\overline\D))$ almost surely on every finite
interval, which proves the regularity.
\end{proof}


\section{Bounds and asymptotic extinction}\label{sec:mass-proof}
\subsection{Finite-time bounds}
\begin{proof}[Proof of Corollary~\ref{cor:tight}]
	With $\eta=\|(v_t)_-\|_\infty$, the event inclusion is
	\[
	\{A_R(T)>K\}
	\subseteq\{\rho_L\le T\}
	\cup\left\{\int_0^{T\wedge\rho_L}\eta(t)\dd t>K\right\}.
	\]
	Indeed, on $\{\rho_L>T\}$,  $\int_0^{T\wedge\rho_L}\eta(t)\dd t=\int_0^{T}\eta(t)\dd t \geq A_R(T)$. H\"older's inequality and Lemma~\ref{lem:a} give
	\begin{align*}
		\E\left(\int_0^{T\wedge\rho_L}\eta(t)\dd t\right)^p
		&\le\frac{T^{p-1}}{(\chi\tau)^p}
		\E\int_0^{T\wedge\rho_L}a(t)^p\dd t\\
		&\le C_T[1+L^p(1+\log L)^p].
	\end{align*}
	Markov's inequality and \eqref{eq:rho-tail} prove \eqref{eq:tight-rate}. First let $K\to\infty$ at fixed $L$, and then let $L\to\infty$, to obtain \eqref{eq:tight}.
\end{proof}

\begin{proof}[Proof of Corollary~\ref{cor:sup-moments}]
	Choose $\lambda\in(s,1)$ and set $X_T=\sup_{0\le t\le T}\|u(t)\|_\infty$.
	For a sufficiently large deterministic $R_0>1$, tail integration yields
	\begin{align*}
		\E X_T^s
		&=s\int_0^\infty R^{s-1}\Prob\left(X_T>R\right)\dd R\\
		&\le R_0^s+C_T\int_{R_0}^\infty
		R^{s-1-\lambda}\bigl[1+(1+\log R)^{\nu_*}\bigr]\dd R<\infty.
	\end{align*}
\end{proof}

\subsection{Almost-sure extinction}

\begin{proof}[Proof of Theorem~\ref{thm:mass}]
	By Lemma~\ref{lem:local-mass} and global existence, $M(t)=\int_\D u(t)$ is a nonnegative continuous local martingale and
	\begin{equation}\label{eq:mass-bracket}
		[M]_\infty=\sigma_0^2\int_0^\infty\left(\int_\D u(t,x)^r\dd x\right)^2\dd t<\infty
		\quad\text{almost surely}.
	\end{equation}	
	The martingale convergence theorem gives $M(t)\to M_\infty\in[0,\infty)$ almost surely. Jensen's inequality implies
	\[
	\int_\D u(t,x)^r\dd x\ge |\D|^{1-r}M(t)^r,
	\qquad
	\int_0^\infty M(t)^{2r}\dd t
	\le\frac{|\D|^{2r-2}}{\sigma_0^2}[M]_\infty<\infty.
	\]
	If $M_\infty>0$, the last integral would diverge. Thus $M_\infty=0$ and \eqref{eq:mass-integral} holds.
	
	For $M_0>0$, the maximal inequality gives
	$\Prob\{\sup_{t\ge0}M(t)>h\}\le\min\{1,M_0/h\}$. Tail integration yields, for $0<s<1$,
	\[
	\E\sup_{t\ge0}M(t)^s
	\le s\int_0^{M_0}h^{s-1}\dd h+sM_0\int_{M_0}^\infty h^{s-2}\dd h
	=\frac{M_0^s}{1-s}.
	\]
	If $M_0=0$, nonnegativity implies $M\equiv0$. In either case, dominated convergence gives $\E M(t)^s\to0$.
	
	Finally, put $N(t)=\int_\D v(t)=\|v(t)\|_1$. For $\tau=0$, integration of the elliptic equation gives $N=M$. For $\tau>0$,
	\begin{equation}\label{eq:detail-signal-mass}
		\tau N'+N=M,\qquad
		N(t)=e^{-t/\tau}N(0)+\tau^{-1}\int_0^t e^{-(t-s)/\tau}M(s)\dd s.
	\end{equation}
	For $t\ge t_0$, splitting at $t_0$ bounds the right-hand side by
	\[
	e^{-t/\tau}N(0)+e^{-(t-t_0)/\tau}\sup_{s\le t_0}M(s)+\sup_{s\ge t_0}M(s).
	\]
	Letting $t\to\infty$ and then $t_0\to\infty$ proves $N(t)\to0$. Moreover,
	\[
	\E\sup_{t\ge0}N(t)^s\le N(0)^s+\E\sup_{t\ge0}M(t)^s<\infty,
	\]
	so dominated convergence also gives $\E N(t)^s\to0$.
\end{proof}


\subsection{Algebraic decay}
\begin{proof}[Proof of Theorem~\ref{thm:mass-decay-rate}]
The mass is a nonnegative continuous local martingale. By
H\"older's inequality,
\[
\frac{\dd[M]_t}{\dd t}
=\sigma_0^2\left(\int_\D u(t,x)^r\dd x\right)^2
\ge a^2M(t)^{\gamma+2},
\qquad a^2=\sigma_0^2|\D|^{-\gamma}.
\]
Assume first that $M_0>0$. For $h\ge0$ and $x>0$, define
\[
F(h,x)=x^s(1+k_shx^\gamma)^{-s/\gamma}.
\]
Direct differentiation gives $F_{xx}<0$ and
\[
\frac{a^2}{2}x^{\gamma+2}F_{xx}(h,x)-F_h(h,x)
=-\frac{a^2s}{2}(\gamma+s)k_sh\,
x^{s+2\gamma}(1+k_shx^\gamma)^{-s/\gamma-2}\le0.
\]
Thus the drift in It\^o's formula for $F(T-t,M(t))$ is
nonpositive, for every fixed $T>0$.

To justify expectations and the behavior at zero, take integers
$n>\max\{M_0,M_0^{-1}\}$ and set
\[
\theta_n=T\wedge\inf\bigl\{t\ge0:
M(t)\notin(n^{-1},n)\ \text{or}\ [M]_t\ge n\bigr\}.
\]
On this stopped interval, $F_x$ is bounded and $[M]_{\theta_n}\le n$,
so the stochastic integral is square integrable. Consequently,
\[
\E F(T-\theta_n,M(\theta_n))\le F(T,M_0).
\]
A nonnegative continuous local martingale is a supermartingale
and is absorbed upon hitting zero. On paths that do not hit zero
by $T$, continuity and finiteness of $[M]_T$ give $\theta_n=T$
for all sufficiently large $n$. On paths that hit zero by $T$,
the upper and quadratic-variation stops eventually disappear,
and
\[
0\le F(T-\theta_n,M(\theta_n))\le n^{-s}\longrightarrow0=M(T)^s.
\]
Fatou's lemma therefore proves \eqref{eq:mass-decay-rate}.

The conditional maximal inequality, first on finite time
intervals and then by monotone convergence, gives
\[
\Prob\left(\sup_{\ell\ge t}M(\ell)>b\mid\mathcal F_t\right)
\le\min\{1,M(t)/b\},\qquad b>0.
\]
Conditional tail integration yields
\[
\E\left[\sup_{\ell\ge t}M(\ell)^s\mid\mathcal F_t\right]
\le\frac{M(t)^s}{1-s}.
\]
Taking expectations proves \eqref{eq:future-mass-supremum}.

For $\tau=0$, integration of the signal equation gives $N=M$.
For $\tau>0$, its mass formula is
\[
N(t)=e^{-t/\tau}N_0+\frac1\tau\int_0^t e^{-(t-\ell)/\tau}M(\ell)\dd\ell.
\]
Splitting the integral at $t/2$ gives
\[
N(t)\le e^{-t/(2\tau)}
\left(N_0+\sup_{0\le\ell\le t/2}M(\ell)\right)
+\sup_{\ell\ge t/2}M(\ell).
\]
Subadditivity of $x\mapsto x^s$, the maximal bound at time zero,
and \eqref{eq:future-mass-supremum} at time $t/2$ prove
\eqref{eq:signal-decay-rate} and the polynomial bound.
If $M_0=0$, the supermartingale property gives $M\equiv0$,
and the remaining assertions follow from the signal mass equation.
\end{proof}

\section*{Acknowledgments}


\section*{Funding}
This work is partially supported by the National Key R\&D Program of China (No. 2022YFA1006001), the National Natural Science Foundation of China (Nos. 12131019, 12571158, 12371151) and the Postdoctoral Fellowship Program (Grade C) of the China Postdoctoral Science Foundation (No. GZC20261645)

\section*{Disclosure statement}
The authors report there are no competing interests to declare.

\section*{Data availability statement}
No datasets were generated or analyzed during the current study.

\appendix
\section{Stochastic convolution}\label{app:stoch}
\begin{proof}[Proof of Lemma~\ref{lem:stoch}]
	Choose $\theta=3/8$ and $\sigma=1/2$. The assumptions on $q$ ensure
	\begin{equation}\label{eq:local-factor-gaps}
		2\theta<1,\qquad\theta-\frac\sigma2-\frac1q>0,
		\qquad\sigma-\frac dq>\alpha.
	\end{equation}
	For $Y_\theta(s)=\int_0^s(s-r)^{-\theta}S(s-r)g(r)\dd W_r$, the $L^q$-valued Burkholder--Davis--Gundy inequality \cite{NVW} and boundedness of $S(t)$ on $L^q$ yield
	\begin{align}
		\E\|Y_\theta(s)\|_{L^q}^q
		&\le C_q\E\left(\int_0^s(s-r)^{-2\theta}
		\|S(s-r)g(r)\|_{L^q}^2\dd r\right)^{q/2}\notag\\
		&\le C_q\E\left(\int_0^s(s-r)^{-2\theta}
		\|g(r)\|_{L^q}^2\dd r\right)^{q/2}.\label{eq:local-inner-bdg}
	\end{align}
	Set
	\[
	k_\theta(h)=\ind_{(0,T)}(h)h^{-2\theta},
	\qquad G(r)=\ind_{(0,T)}(r)\|g(r)\|_{L^q}^2.
	\]
	Young's inequality in time gives
	\begin{align}
		\E\int_0^T\|Y_\theta(s)\|_{L^q}^q\dd s
		&\le C_q\E\|k_\theta*G\|_{L^{q/2}(0,T)}^{q/2}\notag\\
		&\le C_q\|k_\theta\|_{L^1(0,T)}^{q/2}
		\E\int_0^T\|g(r)\|_{L^q}^q\dd r\notag\\
		&\le \frac{C_qT^{(1/2-\theta)q}}{(1-2\theta)^{q/2}}
		\E\int_0^T\|g(r)\|_{L^q}^q\dd r.\label{eq:local-inner-time}
	\end{align}
	
	The factorization formula \cite{DPZ} reads
	\[
	Z(t)=\frac{\sin(\pi\theta)}\pi\int_0^t(t-s)^{\theta-1}S(t-s)Y_\theta(s)\dd s.
	\]
	Semigroup smoothing into $H^{\sigma,q}$ and H\"older's inequality imply
	\begin{align}
		\|Z(t)\|_{H^{\sigma,q}}
		&\le C_T\int_0^t(t-s)^{\theta-1-\sigma/2}\|Y_\theta(s)\|_{L^q}\dd s\notag\\
		&\le C_T\left(\int_0^t h^{(\theta-1-\sigma/2)q'}\dd h\right)^{1/q'}
		\left(\int_0^t\|Y_\theta(s)\|_{L^q}^q\dd s\right)^{1/q}\notag\\
		&\le C_TT^{\theta-\sigma/2-1/q}
		\left(\int_0^T\|Y_\theta(s)\|_{L^q}^q\dd s\right)^{1/q}.
		\label{eq:local-outer-factor}
	\end{align}
Combining \eqref{eq:local-inner-time} and \eqref{eq:local-outer-factor}, we obtain
	\begin{align}
		\E\sup_{t\le T}\|Z(t)\|_{H^{\sigma,q}}^q
		&\le C_TT^{(\theta-\sigma/2)q-1}
		T^{(1/2-\theta)q}\E\int_0^T\|g(s)\|_{L^q}^q\dd s\notag\\
		&=C_TT^{(1/2-\sigma/2)q-1}
		\E\int_0^T\|g(s)\|_{L^q}^q\dd s.\label{eq:local-stoch-sobolev}
	\end{align}
	Since $H^{1/2,q}(\D)\hookrightarrow C^\alpha(\overline\D)$ for $1/2-d/q>\alpha$ \cite{AdamsFournier}, this proves \eqref{eq:local-stoch-general}. For bounded $g$,
	\[
	\E\int_0^T\|g(s)\|_{L^q}^q\dd s
	\le T|\D|M^q,
	\]
	which gives the remaining estimates.
	Finally, let
	\[
	R_\theta f(t)=\int_0^t(t-s)^{\theta-1}S(t-s)f(s)\dd s.
	\]
	Estimate~\eqref{eq:local-outer-factor} makes $R_\theta$ a bounded
	map from $L^q(0,T;L^q)$ into $L^\infty(0,T;H^{1/2,q})$.
	For a step function $f$, its image is continuous in
	$H^{1/2,q}$, since the relevant kernel is integrable at zero.
	Density of step functions and the same estimate then give
	$R_\theta f\in C([0,T];H^{1/2,q})$ for every
	$f\in L^q(0,T;L^q)$.
	Applying this pathwise to $Y_\theta$ and using the factorization
	formula proves the asserted continuity in $C^\alpha$.
	For $0<T\le1$, the semigroup constants can be chosen uniformly,
	which gives the stated small-time estimate with a constant
	independent of $T$.
\end{proof}

\section{Local solution construction}\label{app:local}
\begin{proof}[Proof of Proposition~\ref{prop:local}]
	\emph{Step 1: truncation.}
	For $N>0$, choose a smooth function $\phi_N$ such that
	\[
	\phi_N(z)=z\quad(|z|\le N),\qquad
	z\phi_N(z)\ge0,\qquad
	\|\phi_N\|_\infty+\operatorname{Lip}(\phi_N)<\infty.
	\]
	Set
	\[
	\sigma_N(z)=\sigma_0(\phi_N(z)_+)^r.
	\]
	Then $\sigma_N$ is bounded and globally Lipschitz, agrees with $\sigma_0(z_+)^r$ on $[-N,N]$, and vanishes for $z\le0$. For a continuous adapted $C$-valued process $w$, define
	\[
	v_N[w](t)=
	\begin{cases}
		K\phi_N(w(t)),&\tau=0,\\[2pt]
		S_\tau(t)v_0+\displaystyle\frac1\tau\int_0^tS_\tau(t-s)\phi_N(w(s))\dd s,&\tau>0.
	\end{cases}
	\]
	Elliptic regularity or \eqref{eq:heat-grad} gives
	\begin{align}
		\sup_{t\le T}\|\nabla v_N[w](t)\|_\infty&\le C_{N,T},\notag\\
		\sup_{t\le T}\|\nabla(v_N[w_1]-v_N[w_2])(t)\|_\infty
		&\le C_{N,T}\sup_{t\le T}\|w_1(t)-w_2(t)\|_\infty.
		\label{eq:local-signal-lipschitz}
	\end{align}
The initial gradient is uniformly bounded because
$v_0\in C^2(\overline\D)$ and $\partial_\nu v_0=0$.
For the integral term, the relevant time integral is
$\int_0^t(t-s)^{-1/2}\dd s=2\sqrt t$.
	
	Writing $F_N[w]=\phi_N(w)\nabla v_N[w]$, we have
	\begin{align*}
		F_N[w_1]-F_N[w_2]
		&=(\phi_N(w_1)-\phi_N(w_2))\nabla v_N[w_1]\\
		&\quad+\phi_N(w_2)\nabla(v_N[w_1]-v_N[w_2]),
	\end{align*}
	and hence
	\begin{align}
		\sup_{t\le T}\|F_N[w](t)\|_\infty&\le C_{N,T},\notag\\
		\sup_{t\le T}\|F_N[w_1](t)-F_N[w_2](t)\|_\infty
		&\le C_{N,T}\sup_{t\le T}\|w_1(t)-w_2(t)\|_\infty.
		\label{eq:local-flux-lipschitz}
	\end{align}
	
\emph{Step 2: contraction.}
On the Banach space of adapted continuous $C$-valued processes with norm
\[
\|w\|_{\mathcal X_h}
=\left(\E\sup_{t\le h}\|w(t)\|_\infty^q\right)^{1/q},
\]
consider
\begin{align*}
\mathcal T_Nw(t)
=S(t)u_0-\chi\int_0^tS(t-s)\diver F_N[w](s)\dd s
+\int_0^tS(t-s)\sigma_N(w(s))\dd W_s.
\end{align*}
Since $u_0\in C$, strong continuity and boundedness of the Neumann
heat semigroup give $S(\cdot)u_0\in C([0,h];C)$ and
$\sup_{t\le h}\|S(t)u_0\|_\infty\le\|u_0\|_\infty$.
The bounded flux and noise coefficients imply that the two
convolutions also have continuous $C$-valued paths starting at zero;
the deterministic convolution uses the integrable kernel
$(t-s)^{-1/2}$, and the stochastic convolution uses
Lemma~\ref{lem:stoch}. Thus $\mathcal T_N$ maps this space into itself.
	The deterministic difference satisfies
	\begin{align*}
		\left\|\int_0^\cdot S(\cdot-s)\diver(F_N[w_1]-F_N[w_2])(s)\dd s\right\|_{\mathcal X_h}
		&\le C_{N,T}\int_0^h(h-s)^{-1/2}\dd s\,\|w_1-w_2\|_{\mathcal X_h}\\
		&\le C_{N,T}h^{1/2}\|w_1-w_2\|_{\mathcal X_h}.
	\end{align*}
	For $H=\sigma_N(w_1)-\sigma_N(w_2)$, \eqref{eq:local-stoch-general} gives
	\begin{align}\label{eq:stoch-difference}
		\left(\E\sup_{t\le h}\left\|\int_0^tS(t-s)H(s)\dd W_s\right\|_{C^\alpha}^q\right)^{1/q}
		&\le Ch^{1/4}\left(\E\sup_{s\le h}\|H(s)\|_\infty^q\right)^{1/q}\notag\\
		&\le C_Nh^{1/4}\|w_1-w_2\|_{\mathcal X_h}.
	\end{align}
	Therefore
	\begin{equation}\label{eq:local-contraction}
		\|\mathcal T_Nw_1-\mathcal T_Nw_2\|_{\mathcal X_h}
		\le C_{N,T}(h^{1/2}+h^{1/4})\|w_1-w_2\|_{\mathcal X_h}.
	\end{equation}
	Choose $h\le1$ so that the coefficient is smaller than one. The fixed point is unique on $[0,h]$.
As \eqref{eq:local-signal-lipschitz}--\eqref{eq:local-contraction} retain constants depending only on $N,T$, iteration thus constructs the truncated solution on every finite interval.
	
		\emph{Step 3: energy and positivity.}
	The variational It\^o formula gives the energy identity
	\begin{equation}\label{eq:local-energy-identity}
		\begin{aligned}
			&\|u_N(t)\|_2^2+2\int_0^t\|\nabla u_N(s)\|_2^2\dd s\\
			&\quad=\|u_0\|_2^2+2\chi\int_0^t\langle F_N(s),\nabla u_N(s)\rangle\dd s\\
			&\qquad+\int_0^t\|\sigma_N(u_N(s))\|_2^2\dd s
			+2\int_0^t\langle u_N(s),\sigma_N(u_N(s))\rangle\dd W_s.
		\end{aligned}
	\end{equation}
	Here $F_N=F_N[u_N]$. See \cite{KrylovIto} for the variational formula and smooth approximation method. The Neumann weak form \eqref{eq:local-weak-divergence} accounts for the boundary terms.
	
	Taking expectations yields
	\begin{equation}\label{eq:truncated-energy}
		\E\|u_N(t)\|_2^2+\E\int_0^t\|\nabla u_N(s)\|_2^2\dd s
		\le\|u_0\|_2^2+C_{N,T},\qquad t\le T.
	\end{equation}
	In particular, $u_N\in L^2(\Omega\times(0,T);H^1)$.
	On $\{u_N<0\}$,
	\[
	\sigma_N(u_N)=0,\qquad
	|\phi_N(u_N)|\le\operatorname{Lip}(\phi_N)(u_N)_-.
	\]
	It follows that
	\[
	\begin{aligned}
		&\|(u_N(t))_-\|_2^2+2\int_0^t\|\nabla(u_N(s))_-\|_2^2\dd s\\
		&\quad=-2\chi\int_0^t\langle F_N(s),\nabla(u_N(s))_-\rangle\dd s\\
		&\quad\le\int_0^t\|\nabla(u_N(s))_-\|_2^2\dd s
		+C_{N,T}\int_0^t\|(u_N(s))_-\|_2^2\dd s.
	\end{aligned}
	\]
	The initial negative part is zero and the stochastic terms vanish. Gronwall's inequality gives
	\[
	\E\|(u_N(t))_-\|_2^2=0,
	\qquad u_N\ge0.
	\]
	Because $\phi_N(u_N)\ge0$, positivity of $K$ or $S_\tau(t)$ also gives $v_N\ge0$.

	\emph{Step 4: consistency and continuation.}
	Define
	\[
	\zeta_N=\inf\{t\ge0:\|u_N(t)\|_\infty\ge N\}.
	\]
	For $M>N>\|u_0\|_\infty$, uniqueness and agreement of the coefficients on $[0,N]$ imply
	\[
	(u_M,v_M)=(u_N,v_N)\quad\text{on }[0,\zeta_N],
	\qquad\zeta_N\le\zeta_M.
	\]
	Thus $\zeta=\lim_N\zeta_N$ and $u=u_N$ on $[0,\zeta_N)$ define a maximal local solution. If its density stayed bounded by $N$ before a finite lifetime, it would agree there with the globally constructed truncation at a larger level; continuity of that truncation would extend it beyond the lifetime. Equivalently,
	\[
	\zeta<\infty\quad\Longrightarrow\quad
	\sup_{t<\zeta}\|u(t)\|_\infty=\infty,
	\]
	which is \eqref{eq:alternative}. The stopped identities in Definition~\ref{def:local-solution} follow by choosing a strictly interior localizing sequence.

	\emph{Step 5: spatial and signal regularity.}
	Fix $T>0$ and an integer $n>\|u_0\|_\infty$.
	Set $\vartheta=T\wedge\zeta_n$, where $\zeta_n$ is defined
	in \eqref{eq:local-mass-stops}.
	Then $\vartheta<\zeta$ and
	$\sup_{t\le\vartheta}\|u(t)\|_\infty\le n$ almost surely.
	 The mild formula gives, for $t\in [0,\vartheta]$,
	\begin{align*}
		\|u(t)-S(t)u_0\|_{C^\alpha}
		&\le C_T\int_0^t(t-s)^{-(1+\alpha)/2}\|u(s)\nabla v(s)\|_\infty\dd s\\
		&\quad+\sigma_0\left\|\int_0^tS(t-s)u(s)^r\dd W_s\right\|_{C^\alpha}.
	\end{align*}
	The deterministic kernel is integrable because $(1+\alpha)/2<1$, and the stochastic term is controlled by Lemma~\ref{lem:stoch}.
	Choose $\alpha<\beta_0<1/2-d/q$.
The same estimates give bounded $C^{\beta_0}$ norms for $u-S(\cdot)u_0$ on each localized finite interval. Together with the already established
	$C$-valued time continuity and the interpolation inequality
	\[
	\|f\|_{C^\alpha}\le C\|f\|_\infty^{1-\alpha/\beta_0}\|f\|_{C^{\beta_0}}^{\alpha/\beta_0}.
	\]
	this yields that $u-S(\cdot)u_0\in C([0,\vartheta];C^\alpha)$.
	
	For $\tau>0$, \eqref{eq:signal-derivative} gives, for $t>0$,
	\[
	v_t(t)=-S_\tau(t)Bv_0+\frac{u(t)}\tau
	-\frac1\tau\int_0^tBS_\tau(t-s)u(s)\dd s.
	\]
	For fixed $h>0$, decompose
	$S(t)u_0=S(t)(u_0-S(h)u_0)+S(t+h)u_0$.
	Heat smoothing gives
	\[
	\limsup_{t\downarrow0}t^{\alpha/2}\|S(t)u_0\|_{C^\alpha}
	\le C_T\|u_0-S(h)u_0\|_\infty\longrightarrow0
	\qquad\text{as }h\downarrow0.
	\]
	Since $u-S(\cdot)u_0\in C([0,\vartheta];C^\alpha)$, it follows that
	\[
	\varepsilon(t):=\sup_{0<s\le t\wedge\vartheta}
	s^{\alpha/2}\|u(s)\|_{C^\alpha}\longrightarrow0
	\qquad\text{as }t\downarrow0.
	\]
	Using \eqref{eq:heat-Balpha}, for $0<t\le\vartheta$ we obtain
	\[
	\begin{aligned}
		\left\|\int_0^tBS_\tau(t-s)u(s)\dd s\right\|_\infty
		&\le C_T\varepsilon(t)
		\int_0^t(t-s)^{-1+\alpha/2}s^{-\alpha/2}\dd s\\
		&=C_T\varepsilon(t)\mathrm B(\alpha/2,1-\alpha/2)
		\longrightarrow0\qquad\text{as }t\downarrow0.
	\end{aligned}
	\]
	Here $\mathrm B$ denotes the beta function. The same
	integrable-kernel estimate, together with the local
	$C^\alpha$-valued continuity of $u$ for positive times,
	gives continuity of $v_t$ on compact subintervals of
	$(0,\vartheta]$. Strong continuity of $S_\tau$ and continuity
	of $u$ at zero then yield $v_t(0)=-Bv_0+u_0/\tau$ and
	$v\in C^1([0,\vartheta];C)$. Consequently,
	$\Delta v=\tau v_t+v-u$ is continuous. When $\tau=0$,
	the corresponding conclusion follows directly from
	$\Delta v=v-u$. These regularity properties justify the
	localized calculations used in the main proofs.
\end{proof}

\end{document}